[1996/12/01]
\documentclass{article}
\usepackage{amsmath,amsthm}
\newtheorem{thm}{Theorem}[section]
\newtheorem{lem}[thm]{Lemma}

\theoremstyle{definition}

\let\Lm=\Lambda
\let\vph=\varphi

\let\abs=\envert

\theoremstyle{remark}

\newtheorem{rem}[thm]{Remark}

\title{Quasiperfect numbers with the same exponent
}

\date{}
\author{Tomohiro Yamada}
\begin{document}
\maketitle

\begin{abstract}
We study some divisibility properties of quasiperfect numbers.  We show that
if $N=(p_1 p_2 \cdots p_t)^{2a}=m^2$ is quasiperfect, then
$2a+1$ is divisible by $3$ and $N$ has at least one prime factor
smaller than $\exp 716.7944$.  Moreover,
we find some lower bounds concerning quasiperfect numbers of the form $N=m^2$
with $m$ squarefree.
\end{abstract}

\section{Introduction}\label{intro}
A positive integer $N$ is called to be perfect if $\sigma(N)=2N$,
where $\sigma(N)$ denotes the sum of divisors of $N$.
As is well known, an even integer $N$ is perfect if and only if
$N=2^{k-1}(2^k-1)$ with $2^k-1$ prime.
In contrast, it is one of the oldest unsolved problems whether there exists
an odd perfect number or not.  Moreover, it is also unknown that
there exists an odd multiperfect number, an integer $N$ with $\sigma(N)=kN$ for some integer $k\geq 2$.

Cattaneo\cite{Cat} called a positive integer $N$ quasiperfect if $\sigma(N)=2N+1$
and showed that such an integer must be an odd square, any proper divisor $m$ of $N$
satisfies $\sigma(m)<2m$ and any divisor of $\sigma(N)$ is congruent to $1$ or $3$ modulo $8$.
Hagis and Cohen\cite{HC} showed that if $N$ is quasiperfect, then
$N>10^{35}$ and $N$ has at least $7$ distinct prime factors.

It is well-known that an odd perfect number must be of the form $q^b p_1^{2a_1}p_2^{2a_2}\cdots p_t^{2a_t}$
for some integers $a_1, a_2, \ldots, a_t, b$ and distinct primes $p_1, p_2, \ldots, p_t, q$
with $q\equiv b\equiv 1\pmod{4}$.
In the special case $a_1=a_2=\cdots =a_t=a$, several results are known as follows:
\begin{itemize}
\item[1)] If $a_1=a_2=\cdots =a_t=a$, we know that $a\geq 9$ and $a\neq 10, 11, 12, 13, 14, 16, 17, 18, 19, 24, 62$,
combining results in \cite{CW}\cite{HMD}\cite{Kan}\cite{Mc}\cite{MDH}\cite{St}.
\item[2)] If $d$ divides $2a_i+1$ for all $i$, then $d$ cannot be one of $3, 35, 65$,
combining results in \cite{EP}\cite{Mc}\cite{MDH}.
\item[3)] The author showed that, for any given $a$, there exist only finitely many odd perfect numbers
of the form $q^b (p_1 p_2\cdots p_t)^{2a}$ \cite{Ymd1}.
\item[4)] Fletcher, Nielsen and Ochem\cite{FNO} shows that
an odd perfect number $q^b p_1^{2a_1}p_2^{2a_2}\cdots p_t^{2a_t}$
for which there exists a finite set $S$ of primes such that each $2a_i+1$ divisible by a prime in $S$
must have a prime factor below an effectively computable constant $C$ depending on $S$.
In \cite{Ymd3}, the author gave an explicit upper bound for $C$.
\end{itemize}

Some results similar to 1), 2) above are known for quasiperfect numbers of the form
$(p_1 p_2 \cdots p_t)^{2a}$ for a given integer $a$.
Cohen\cite{Coh} showed that $a$ must be congruent to $1, 3, 5, 9$ or $11 \pmod{12}$.
Moreover, if an integer of the form $p_1^{6a_1+2}p_2^{6a_2+2}\cdots p_t^{6a_t+2}$
is quasiperfect, then $t\geq 230876$.  We would like to begin by extending this result,
which follows from some elementary consideration.
\begin{thm}\label{thm1}
If $N=p_1^{2a_1}p_2^{2a_2}\cdots p_t^{2a_t}=m^2$ is quasiperfect,
then there exists a prime factor $p_j\equiv 1\pmod{4}$ for which $2a_j+1$ does not have a divisor congruent to $5\pmod{8}$.
Moreover, if $N=(p_1 p_2 \cdots p_t)^{2a}=m^2$ is quasiperfect, then
\begin{itemize}
\item[a)] All prime factors of $N$ must be congruent to $1$ or $7\pmod{8}$.
\item[b)] $2a+1\equiv 3\pmod{8}$ and all prime factors of $2a+1$ must be congruent to $1$ or $3\pmod{8}$.
\item[c)] $2a+1$ must be divisible by $3$.
\end{itemize}
\end{thm}

\begin{proof}
Since $\sigma(N)=2N+1=2m^2+1\equiv 3\pmod{8}$, we have $\sigma(p_j^{2a_j})\not\equiv 1\pmod{8}$
for some $j$, for which, by Cattaneo's result mentioned above,
$\sigma(p_j^{2a_j})\equiv 3\pmod{8}$ cannot have a prime factor
congruent to $5$ or $7$ modulo $8$.
If $d$ divides $2a_j+1$, then $\sigma(p_j^{d-1})$ divides $\sigma(p_j^{2a_j})\mid\sigma(N)=2m^2+1$
and therefore has only prime factors congruent to $1$ or $3$ modulo $8$.
Since $\sigma(p_j^{2a_j})\equiv 3\pmod{8}$, we must have $p_j\equiv 1\pmod{4}$.
If $p_j\equiv 1\pmod{8}$, then $d\equiv 1$ or $3\pmod{8}$ and,
if $p_j\equiv 5\pmod{8}$, then $d\equiv 1$ or $7\pmod{8}$.
Hence $d$ cannot be congruent to $5$ modulo $8$.  This proves the former part of the theorem.

Next, assume that $N=(p_1 p_2 \cdots p_t)^{2a}=m^2$ is quasiperfect.
By Cohen's result mentioned above, we must have $2a+1\equiv 3\pmod{4}$.
Hence $N$ has no prime factor congruent to $3\pmod{8}$.
In particular, $N=m^2$ is not divisible by $3$ and therefore $\sigma(N)=2m^2+1$ must be divisible by $3$.
Hence there exists a prime factor $p_k$ for which $\sigma(p_k^{2a})$ is divisible by $3$.
This implies $p_k\equiv a\equiv 1\pmod{3}$.  So that $2a+1$ must be divisible by $3$
as stated in c).
Since $p_i^2+p_i+1$ divides $\sigma(N)$, $p_i^2+p_i+1$ must be congruent
to $1$ or $3$ modulo $8$ and therefore $p_i\equiv \pm 1\pmod{8}$ for any $i$,
proving a).  By the former part of the theorem, $p_s\equiv 1\pmod{8}$ for some $s$.
If $d$ divides $2a_s+1$, then $\sigma(p_s^{d-1})\mid\sigma(p_s^{2a})\mid\sigma(N)=2m^2+1$
and therefore $d\equiv 1$ or $3\pmod{8}$.
Finally, Cohen's result mentioned above yields that $a$ cannot be a multiple of $4$
and therefore $2a+1\equiv 3\pmod{8}$, which proves b).
\end{proof}

No similar result to 3) has been known for quasiperfect numbers and
neither has the author been able to prove such a result.  Instead,
in \cite{Ymd3}, the author proved that,
if $N=p_1^{2a_1}p_2^{2a_2}\cdots p_t^{2a_t}$ is quasiperfect
and there exists a finite set $S$ of primes such that each $2a_i+1$ is divisible by a prime in $S$,
then $N$ must have a prime factor below $C_0=\exp 2173.5\abs{S}^2\max\{8l, \exp 13.3\}$.
More generally, the author proved that, if $N=p_1^{2a_1}p_2^{2a_2}\cdots p_t^{2a_t}$
satisfies that $\sigma(N)\geq 2N$ has no prime factor congruent to $5$ or $7$ modulo $8$
and there exists a finite set $S$ of primes such that each $2a_i+1$ is divisible by a prime in $S$,
then $N$ must have a prime factor below $C_0$.  In this paper, we prove the following result.
\begin{thm}\label{thm2}
If $N=(p_1 p_2 \cdots p_t)^{2a}=m^2$ is quasiperfect, then
$N$ must have a prime factor below $C=\exp 716.7944<1.995 \cdot 10^{311}$.
\end{thm}

We shall give the outline of our proof here.
From Theorem \ref{thm1}, we can see that, each $p_i^2+p_i+1$ divides $\sigma(N)=2m^2+1$
and therefore cannot have a prime factor $\equiv 7, 13 \pmod{24}$,
which is implicit in the Note of Lemma 3 of \cite{Coh}.
Using some sieve argument, we shall prove that the number of prime $p\leq X$ such that
$p^2+p+1$ has no prime factor $\equiv 7, 13 \pmod{24}$ is $<cX/\log^{3/2} X$ with an explicit constant $c$.

Our sieve argument is based on the author's one used in \cite{Ymd2}
to prove that an odd perfect number $q^b p_1^{2a_1}p_2^{2a_2}\cdots p_t^{2a_t}$
with all $p_i$ in a given finite set $S$ must have a prime factor
below an effectively computable constant $C$ depending on $S$.
This method was refined by Fletcher, Nielsen and Ochem\cite{FNO} and the author\cite{Ymd3}
to prove 4).

Finally, we would like to show the following lower bounds concerning
quasiperfect number of the form $N=m^2$ with $m$ squarefree.
\begin{thm}\label{thm3}
If $N=m^2$ with $m$ squarefree is quasiperfect, then
$N>\exp 17840573219$ and $N$ must have at least $406550054$ distinct prime factors,
one of which is $\geq 9457308739$.
\end{thm}

Our results supports the conjecture that there exists no quasiperfect number
of the form $N=(p_1 p_2 \cdots p_t)^{2a}$.

\section[Proof of Theorem 1.2]{Proof of Theorem \ref{thm2}}

From a) of Theorem \ref{thm1} and a remark in the introduction we see that,
if $p$ is a prime factor of $N$, then $p\equiv \pm 1\pmod{8}$ and
$p^2+p+1$ cannot have a prime factor congruent to $7$ or $13 \pmod{24}$.
Letting $Q^\pm$ be the set of prime numbers $p\equiv \pm 1\pmod{8}$ such that
$p^2+p+1$ has no prime factor congruent to $7$ or $13 \pmod{24}$,
we see that any prime factor of $N$ must be contained in either $Q^+$ or $Q^-$.
Hence it suffices to show that $\prod_{p\geq C, p\in Q^+\cup Q^-}\frac{p}{p-1}<2$.

We would like to introduce some notations in order to apply sieve methods.
First, $\pm$ may take a different sign in each different occurence
but shall take the same sign in the same context.
Let $\Omega_p$ be a set of congruent classes modulo each prime $p$ and
define $\rho(p)$ to be the number of such congruent classes and
$S(\Omega)=S(x, \Omega, y)$ to be the set of integers $\leq x$ which does not belong to $\Omega_p$ for any prime $p\leq y$.

Putting $\Omega^\pm _p=\{n\mid (8n\pm 1)((8n\pm 1)^2+(8n\pm 1)+1)\equiv 0\pmod{p}\}$
for primes $p$ congruent to $7$ or $13 \pmod{24}$
and $\Omega^\pm_p=\{n\mid 8n\pm 1\equiv 0\pmod{p}\}$ for other odd primes $p$,
we have $\rho^\pm (2)=0$, $\rho^\pm (p)=3$ for primes $p\equiv 7, 13\pmod{24}$
and $\rho^\pm(p)=1$ for the other primes $p$.
Moreover, it is clear that, if $p=8n\pm 1$ belongs to $Q^\pm$ and $p\geq y, n\leq x$,
then $n$ must be contained in $S(x, \Omega^\pm, y)$ and therefore
\begin{equation}\label{eq21}
\pi^\pm (8x\pm 1)\leq y+\abs{S(x, \Omega^\pm, y)}
\end{equation}
where $\pi^\pm(X)$ denotes the number of primes $\leq X$ belonging to $Q^\pm$.

In order to estimate $\abs{S(x, \Omega^\pm, y)}$, we use the sieve method
mentioned in introduction.  Let us introduce further notations
\begin{equation}
B(z)=B^\pm (z)=\frac{1}{\log z}\sum_{p\leq z} \frac{\rho^\pm (p)\log p}{p},
\end{equation}
\begin{equation}
V(z)=V^\pm (z)=\prod_{p\leq z}\left(1-\frac{\rho^\pm (p)}{p}\right)
\end{equation}
and
\begin{equation}
\psi_0^\pm (v, u)=1-\exp(-\psi_1(B^\pm (x^\frac{1}{v}), v/u)),
\end{equation}
where
\begin{equation}
\psi_1(K, t)=\max\left\{0, t\log \frac{t}{K}-t+K\right\}.
\end{equation}

Now, we have the following sieve inequality.
\begin{lem}\label{lm21}
For any real $v\geq u\geq 2$, we have
\begin{equation}\label{eq22}
\abs{S(x, \Omega^{\pm}, x^\frac{1}{u})}\leq \frac{(x+x^\frac{2}{u})V^\pm(x^{1/u})}{\psi_0^\pm (v, u)}.
\end{equation}
\end{lem}
\begin{proof}
By definition, we have $\rho^\pm(p)<p$ for any prime $p$ and therefore,
we can apply Theorem 7.14 in \cite{IK} and Lemma 2.2 of \cite{Ymd3} to obtain (\ref{eq22}).
\end{proof}

Now we need to estimate $B(z)$ and $V(z)$.  To this end, we need some explicit estimates
for the number of primes in arithmetic progressions.
Our start points are the following error estimates in the prime number theorem for arithmetic progressions with difference $24$,
which can be obtained from some known explicit versions of the prime number theorem for arithmetic progressions
and a recent numerical results for the Generalized Riemann Hypothesis of $L$-functions.
\begin{lem}\label{lm31}
As in \cite{Dus1}, we set $R=9.645908801$.
Let $l$ be any integer coprime to $6$.  Then,
\begin{equation}\label{eq31a}
\frac{1}{z}\abs{\psi(z; 24, l)-\frac{z}{8}}<
\begin{cases}c_1:=6.27961\cdot 10^{-4} & \text{ for } z\geq 10^{10}, \\
c_2:=1.94638\cdot 10^{-5} & \text{ for } z\geq e^{30}, \\
c_3:=2.4432\cdot 10^{-7} & \text{ for } z\geq e^{60}. \\
\end{cases}
\end{equation}

Moreover, for any real $z\geq e^{625R}$, we have
\begin{equation}\label{eq31b}
\abs{\psi(z; 24, l)-\frac{z}{8}}<\frac{3.6\cdot 10^{-7}z}{\log z}
\end{equation}
and
\begin{equation}\label{eq31c}
\abs{\psi(z; 24, l)-\frac{z}{8}}<\frac{0.0022z}{\log^2 z}.
\end{equation}
\end{lem}

\begin{proof}
Platt \cite{Pla} confirmed that $L(s, \chi)$ has no nontrivial zeros with $\Re s\neq 1/2$ and $\abs{\Im s}\leq 10^8/24$
for all characters $\chi$ modulo $24$.
Now (\ref{eq31a}) follows from Theorem 4.3.2 of \cite{RR} with $H=10^8/24, C_1=38.31$.

Moreover, from Theorem 3.6.3 of \cite{RR}, we see that
any nontrivial zero $s$ of $L(s, \chi)$ satisfies $1-\Re s\geq 1/R\log (24\abs{\Im s}/38.31)$
for all characters $\chi$ modulo $24$.
Now we can apply Theorem 5 of \cite{Dus1} with $H=10^8/24, C_1(k)=38.31$ and $X_4\leq 25$ for any characters modulo $24$
to obtain (\ref{eq31b}) and (\ref{eq31c}).
This proves the lemma.
\end{proof}

\begin{rem}
Kadiri \cite{Kad} claimed to prove that
any nontrivial zero $s$ of $L(s, \chi)$ satisfies $1-\Re s\geq 1/6.397\log (24\abs{\Im s})$
for all characters $\chi$ modulo $24$.  This would give better estimates.
\end{rem}

Using these error estimates, we obtain the following bounds.
\begin{lem}\label{lm32}
For any real $z>e^{60}$, we have
\begin{equation}\label{eq32a}
\sum_{p\leq z, p\equiv 7, 13\pmod{24}}\frac{\log p}{p}<\frac{\log z}{4}
\end{equation}
and
\begin{equation}\label{eq32b}
\prod_{p\leq z, p\equiv 7, 13\pmod{24}}\left(1-\frac{1}{p}\right)^{-1}>0.95442 \log^\frac{1}{4} z.
\end{equation}
\end{lem}

\begin{proof}
In this proof of the lemma, we let $l$ be an integer congruent to $7$ or $13$ modulo $24$.
We begin by proving (\ref{eq32a}) for $z\geq e^{60}$.
Lemma \ref{lm31} gives
\begin{equation}
\begin{split}
\int_{10^{10}}^{e^{625R}} \frac{\abs{E_\psi(t; 24, l)}}{t^2} dt
< & (30-10\log 10)c_1+30c_2+(625R-60)c_3 \\
< & 0.0065,
\end{split}
\end{equation}
where $E_\psi(t; k, l)=\psi(t, k, l)-t/\vph(k)$.
By Theorem 5.2.1 of \cite{RR}, we have
\begin{equation}
\int_{10^8}^{10^{10}} \frac{\abs{E_\psi(t; 24, l)}}{t^2} dt
<1.745\int_{10^8}^{10^{10}}\frac{dt}{t^{3/2}}=\frac{3.49}{10^4}-\frac{3.49}{10^5}<0.00032.
\end{equation}
Some calculations yield that
\begin{equation}
\sum_{p<10^8, p\equiv 7\pmod{24}}\frac{\log p}{p}<\log 10-0.101846,
\end{equation}
and
\begin{equation}
\sum_{p<10^8, p\equiv 13\pmod{24}}\frac{\log p}{p}<\log 10-0.202137.
\end{equation}
Moreover, other calculations give $\psi(10^8; 24, 7)>12499496$ and $\psi(10^8; 24, 13)>12499441$.
We use partial summation and apply these inequalities and (\ref{eq31c}) to obtain that,
for $z\geq 10^8$ and $l=7$ or $13$,
\begin{equation}
\begin{split}
\sum_{\substack{p\leq z,\\ p\equiv l\pmod{24}}}\frac{\log p}{p}
\leq & \sum_{\substack{p\leq 10^8,\\ p\equiv l\pmod{24}}}\frac{\log p}{p}+\sum_{\substack{10^8<n\leq z,\\ n\equiv l\pmod{24}}}\frac{\Lm(n)}{n} \\
= & \sum_{\substack{p\leq 10^8,\\ p\equiv l\pmod{24}}}\frac{\log p}{p}+\frac{\psi(z; 24, l)}{z}-\frac{\psi(10^8; 24, l)}{z}+\int_{10^8}^z \frac{\psi(t; 24, l)}{t^2} dt\\
< & \frac{\log z}{8}-0.101846+\frac{0.01}{\log z}+0.0004+0.0065 \\
< & \frac{\log z}{8},
\end{split}
\end{equation}
which immediately gives (\ref{eq32a}) for $z\geq e^{60}$ (and even for $z\geq 10^8$).

Next we prove (\ref{eq32b}) for $z\geq e^{60}$.
By (\ref{eq31a}), for any real $z$ with $e^{60}\leq z<e^{625R}$, we have
\begin{equation}\label{eq33a}
\int_{z}^{e^{625R}} \frac{c_3(1+\log t)}{t\log^2 t}dt=c_3\left(\frac{1}{\log z}-\frac{1}{625R}+\log (625R)-\log\log z\right)
\end{equation}
and, by (\ref{eq31b}), for any real $z\geq e^{625R}$, we have
\begin{equation}
\int_{z}^{\infty} \frac{3.6\times 10^{-7}(1+\log t)}{t\log^3 t}dt=3.6\times 10^{-7}\left(\frac{1}{\log z}+\frac{1}{2\log^2 z}\right).
\end{equation}

Henceforth, we let $z\geq e^{60}$ as in the lemma.
Using the above estimates, we have
\begin{equation}
\begin{split}
& \sum_{\substack{p\leq z,\\ p\equiv l\pmod{24}}} \frac{1}{p}\\
= & \frac{1}{8}\log\log z+M(24, l)+\frac{E_\theta(z; 24, l)}{z\log z}-\int_z^\infty \frac{(1+\log t)E_\theta(t; 24, l)}{t^2\log^2 t}dt \\
> & \frac{1}{8}\log\log z+M(24, l)-1.1304\times 10^{-6}-\frac{3.73844\times 10^{-7}}{\log z},
\end{split}
\end{equation}
where $E_\theta(z; k, l)=\theta(z; k, l)-z/\vph(k)$ and $M(k, l)$ denotes the limit
\begin{equation}
\lim_{x\rightarrow\infty}\sum_{p\leq x, p\equiv l\pmod{k}}\frac{1}{p}-\frac{\log\log x}{\vph(k)}.
\end{equation}
Using $M(24,7) = 0.003897\cdots$ and $M(24,13) = -0.0681541\cdots$, as is given by \cite{LZ},
we obtain
\begin{equation}
\sum_{\substack{p\leq z,\\ p\equiv 7, 13\pmod{24}}} \frac{1}{p}>\frac{1}{8}\log\log z-0.06426
\end{equation}
and therefore
\begin{equation}
\begin{split}
\prod_{\substack{p\leq z,\\ p\equiv 7, 13\pmod{24}}}\left(1-\frac{1}{p}\right)^{-1}
= & \exp\sum_{\substack{p\leq z,\\ p\equiv 7, 13\pmod{24}}}\left(\frac{1}{p}+\frac{1}{2p^2}+\cdots \right) \\
> & \exp 0.015815+\sum_{\substack{p\leq z,\\ p\equiv 7, 13\pmod{24}}}\frac{1}{p} \\
> & 0.95442\log^\frac{1}{4} z,
\end{split}
\end{equation}
which proves the lemma.
\end{proof}

Now it is easy to estimate $V(P(z))$ and $B(z)$ using Lemma \ref{lm32} and known error estimates
for the ordinary prime number theorem.
Combining (\ref{eq32b}) and the estimate
\begin{equation}
\prod_{p\leq z}\left(1-\frac{1}{p}\right)<\frac{e^{-\gamma}}{\log z}\left(1+\frac{1}{5\log^3 z}\right)
\text{ for } z\geq 2278382
\end{equation}
given in Theorem 5.9 of \cite{Dus2}, we conclude that
\begin{equation}\label{eq23a}
\begin{split}
V(P(z))< & \prod_{2<p\leq z}\left(1-\frac{1}{p}\right) \prod_{\substack{p\leq z,\\ p\equiv 7, 13\pmod{24}}}\left(1-\frac{1}{p}\right)^2 \\
< & \frac{2e^{-\gamma}}{(0.95442)^2\log^\frac{3}{2} z}\left(1+\frac{1}{5\log^2 z}\right) \\
< & \frac{1.23274}{\log^\frac{3}{2} z}\text{ for } z>e^{40}.
\end{split}
\end{equation}
Moreover, by (3.22) in p. 70 of \cite{RS}, we have
\begin{equation}
\sum_{p\leq z} \frac{\log p}{p}<\log z-1.33258+\frac{1}{2\log z}<\log z
\end{equation}
for $z>319$ and, combining with (\ref{eq32a}), we immediately obtain
\begin{equation}\label{eq23b}
\sum_{p\leq z} \frac{\rho(p)\log(p)}{p}<\frac{3}{2}\log z
\end{equation}
or, equivalently, $B(z)<1.5$ for $z\geq e^{40}$.

(\ref{eq23a}) and (\ref{eq23b}) allow us to take $B=1.5, u=2.0174$ and $v=7.58$ in the sieve inequality (\ref{eq22})
and we obtain
\begin{equation}
\abs{S(x, \Omega^{\pm}, y)}\leq \frac{37.00754x}{\log^\frac{3}{2} x}
\end{equation}
for $x>\exp 714$ and therefore, by (\ref{eq21}),
\begin{equation}
\pi^\pm (X)\leq \frac{4.63941X}{\log^\frac{3}{2} X}
\end{equation}
for $X>\exp 716.5$.

Now we have
\begin{equation}
\begin{split}
\prod_{p\geq C, p\in Q^+\cup Q^-}\frac{p}{p-1}< & \exp \sum_{p\geq C, p\in Q^+\cup Q^-} \left(\frac{1}{p}+\frac{1}{2p^2}+\cdots \right) \\
< & \exp \left(\frac{1}{C}+\sum_{p\geq C, p\in Q^+\cup Q^-} \frac{1}{p}\right) \\
< & \exp \frac{2}{C}+\int_{C}^\infty\frac{9.27882 dt}{t\log^\frac{3}{2} t} \\
= & \exp \frac{2}{C}+\frac{18.55764}{\log^\frac{1}{2} C}<2,
\end{split}
\end{equation}
recalling that $C=\exp 716.7944$.  This proves Theorem \ref{thm2}.

\section[Proof of Theorem 1.3]{Proof of Theorem \ref{thm3}}
Assume that $N=m^2$ with $m$ squarefree is quasiperfect and
let $P$ be the greatest prime factor of $N$.  As noted in the previous section,
any prime factor of $N$ belongs to $Q^+\cup Q^-$.

Calculation gives
\begin{equation}\label{eq41}
\prod_{p<2^{29}, p\in Q^+\cup Q^-}\frac{p^2+p+1}{p^2}<1.75014319434,
\end{equation}
\begin{equation}\label{eq42}
\prod_{p<2^{29}, p\in Q^+\cup Q^-}p>\exp 62460825.5
\end{equation}
and
\begin{equation}\label{eq43}
\pi^+(2^{29})+\pi^-(2^{29})=3285696.
\end{equation}
By Theorem 6.12 of \cite{Dus2}, we have
\begin{equation}
\prod_{2^{29}<p\leq P}\frac{p}{p-1}<\frac{\log P}{29\log 2}\left(1+\frac{1}{121945\log^3 2}\right)\left(1+\frac{1}{5\log^3 P}\right).
\end{equation}
Combined with (\ref{eq41}), we obtain
\begin{equation}
\prod_{p\leq 9457308738, p\in Q^+\cup Q^-}\frac{p^2+p+1}{p^2}<2
\end{equation}
and therefore we must have $P\geq 9457308739$.

Now, let $P_0=9457308739$.  Then we must have
\begin{equation}\label{eq42}
m\geq \left(\prod_{p<2^{29}, p\in Q^+\cup Q^-}p\right) \left(\prod_{2^{29}<p\leq P_0}p\right).
\end{equation}
Theorem 4.2 of \cite{Dus2} gives
\begin{equation}
\prod_{2^{29}<p\leq P_0}p=\exp(\theta(P)-\theta(2^{29}))>\exp\left(P_0-\frac{P_0}{100\log^2 P_0}-536842885.9\right)
\end{equation}
and, with the aid of (\ref{eq42}), we see that $m>\exp 8920286609.5$ and $N>\exp 17840573219$.

Finally, we observe that, using (\ref{eq43}) and Theorem 5.1 of \cite{Dus2},
\begin{equation}
\begin{split}
& \pi^+(P)+\pi^-(P)\geq \pi(P_0)-\pi(2^{29})+\pi^+(2^{29})+\pi^-(2^{29}) \\
\geq & \frac{P}{\log P}\left(1+\frac{1}{\log P}+\frac{2}{\log^2 P}+\frac{7.32}{\log^3 P}\right)-\pi(2^{29})+3285696>406550053.02
\end{split}
\end{equation}
and therefore $N$ must have at least $406550054$ prime factors.  This completes the proof of Theorem \ref{thm3}.

{}
\vskip 12pt

{\small Tomohiro Yamada}\\
{\small Center for Japanese language and culture\\Osaka University\\562-8558\\8-1-1, Aomatanihigashi, Minoo, Osaka\\Japan}\\
{\small e-mail: \protect\normalfont\ttfamily{tyamada1093@gmail.com}}
\end{document}